\documentclass{amsart}

\usepackage{graphicx}
\usepackage{graphics}
\usepackage{graphics}
\usepackage{amsmath,amsthm,verbatim,amssymb,amsfonts,amscd, graphicx}

\usepackage{tikz}
\usepackage{parskip}
\usepackage{float}
\newtheorem{theorem}{Theorem}[section]

\theoremstyle{definition}

\newtheorem{corollary}{Corollary}
\theoremstyle{remark}
\newtheorem{remark}[theorem]{Remark}

\numberwithin{equation}{section}

\begin{document}

\title{On a certain identity involving the Gamma function}

\author{Theophilus Agama}
\address{Department of Mathematics, African Institute for Mathematical science, Ghana
}
\email{theophilus@aims.edu.gh/emperordagama@yahoo.com}



\date{\today}



\footnote{
\par
And here is the beginning of the second paragraph.}%
.

\maketitle

\begin{abstract}
The goal of this paper is to prove the identity \begin{align}\sum \limits_{j=0}^{\lfloor s\rfloor}\frac{(-1)^j}{s^j}\eta_s(j)+\frac{1}{e^{s-1}s^s}\sum \limits_{j=0}^{\lfloor s\rfloor}(-1)^{j+1}\alpha_s(j)+\bigg(\frac{1-((-1)^{s-\lfloor s\rfloor +2})^{1/(s-\lfloor s\rfloor +2)}}{2}\bigg)\nonumber \\ \bigg(\sum \limits_{j=\lfloor s\rfloor +1}^{\infty}\frac{(-1)^j}{s^j}\eta_s(j)+\frac{1}{e^{s-1}s^s}\sum \limits_{j=\lfloor s\rfloor +1}^{\infty}(-1)^{j+1}\alpha_s(j)\bigg)=\frac{1}{\Gamma(s+1)},\nonumber
\end{align}where \begin{align}\eta_s(j):=\bigg(e^{\gamma (s-j)}\prod \limits_{m=1}^{\infty}\bigg(1+\frac{s-j}{m}\bigg)\nonumber \\e^{-(s-j)/m}\bigg)\bigg(2+\log s-\frac{j}{s}+\sum \limits_{m=1}^{\infty}\frac{s}{m(s+m)}-\sum \limits_{m=1}^{\infty}\frac{s-j}{m(s-j+m)}\bigg), \nonumber
\end{align}and \begin{align}\alpha_s(j):=\bigg(e^{\gamma (s-j)}\prod \limits_{m=1}^{\infty}\bigg(1+\frac{s-j}{m}\bigg)e^{-(s-j)/m}\bigg)\bigg(\sum \limits_{m=1}^{\infty}\frac{s}{m(s+m)}-\sum \limits_{m=1}^{\infty}\frac{s-j}{m(s-j+m)}\bigg),\nonumber
\end{align}where $\Gamma(s+1)$ is the Gamma function defined by $\Gamma(s):=\int \limits_{0}^{\infty}e^{-t}t^{s-1}dt$ and $\gamma =\lim \limits_{n\longrightarrow \infty}\bigg(\sum \limits_{k=1}^{n}\frac{1}{k}-\log n\bigg)=0.577215664\cdots $ is the Euler-Mascheroni constant.
\end{abstract}
\bigskip

\section{INTRODUCTION}
The Euler-Gamma function is defined by, $\Gamma(s):=\int \limits_{0}^{\infty}e^{-t}t^{s-1}dt$, valid in the entire complex plane, except at $s=0, -1, -2, \ldots$ where it has simple poles \cite{sebah2002introduction}. It can also be seen as a generalization of the factorial on the positive integers to the rationals. Indeed the Gamma function (See \cite{batir2017bounds}, \cite{sebah2002introduction}) satisfies the functional equation $\Gamma(1)=1$ and \begin{align}\Gamma(s)=\frac{\Gamma(s+n)}{s(s+1)(s+2)\cdots (s+n-1)}\nonumber
\end{align}so that in the case $s=1$ and $n$ is a positive integer, then we have the expression $\Gamma(n+1)=1\cdot 2\cdots n=n!$.  The Gamma function still remains valid for arguments in the range $-1<s<0$ by the equation \begin{align}\Gamma(s)=\frac{\Gamma(s+1)}{s}.\nonumber 
\end{align}It also has the canonical product representation (See \cite{nantomah2015some})\begin{align}\Gamma(s+1)=e^{-\gamma s}\prod \limits_{m=1}^{\infty}\bigg(\frac{m}{m+s}\bigg)e^{s/m},\nonumber
\end{align}valid for $s>-1$. The gamma function also has very key properties, most notably the duplication and the complementary property (reflexive formula), which are given respectively as \begin{align}\Gamma(x)\Gamma(1-x)=\frac{\pi}{\sin \pi x},\nonumber 
\end{align}and \begin{align}\Gamma(x)\Gamma(x+1/2)=\frac{\sqrt{\pi}}{2^{2x-1}}\Gamma(2x).\nonumber
\end{align}For many more of these  properties, the reader is encouraged to see \cite{sebah2002introduction}. The Gamma function is also inextricably linked to some very interesting functions. Consider the digamma function \cite{sebah2002introduction}, the logarithmic derivative of the Gamma function defined by \begin{align}\Psi(x):=\frac{\Gamma'(x)}{\Gamma(x)}=-\gamma+\sum \limits_{m=1}^{\infty}\frac{(x-1)}{m(m+x-1)}.\nonumber
\end{align}The Gamma funtion has spawn a great deal of research and out of which has led to the discovery of many beautiful identities and inequalities. More recently the gamma function has been studied by Alzer and many other authors. For more results on the gamma function, see \cite{batir2017bounds}, \cite{nantomah2015some}.  In this paper, however, we prove a certain identity related to the Gamma function.
 
\section{MAIN THEOREM}
\begin{theorem}\label{universal}
For any $s>1$, we have \begin{align}\sum \limits_{j=0}^{\lfloor s\rfloor}\frac{(-1)^j}{s^j}\eta_s(j)+\frac{1}{e^{s-1}s^s}\sum \limits_{j=0}^{\lfloor s\rfloor}(-1)^{j+1}\alpha_s(j)+\bigg(\frac{1-((-1)^{s-\lfloor s\rfloor +2})^{1/(s-\lfloor s\rfloor +2)}}{2}\bigg)\nonumber \\ \bigg(\sum \limits_{j=\lfloor s\rfloor +1}^{\infty}\frac{(-1)^j}{s^j}\eta_s(j)+\frac{1}{e^{s-1}s^s}\sum \limits_{j=\lfloor s\rfloor +1}^{\infty}(-1)^{j+1}\alpha_s(j)\bigg)=\frac{1}{\Gamma(s+1)},\nonumber
\end{align}where \begin{align}\eta_s(j):=\bigg(e^{\gamma (s-j)}\prod \limits_{m=1}^{\infty}\bigg(1+\frac{s-j}{m}\bigg)\nonumber \\e^{-(s-j)/m}\bigg)\bigg(2+\log s-\frac{j}{s}+\sum \limits_{m=1}^{\infty}\frac{s}{m(s+m)}-\sum \limits_{m=1}^{\infty}\frac{s-j}{m(s-j+m)}\bigg), \nonumber
\end{align}and \begin{align}\alpha_s(j):=\bigg(e^{\gamma (s-j)}\prod \limits_{m=1}^{\infty}\bigg(1+\frac{s-j}{m}\bigg)e^{-(s-j)/m}\bigg)\bigg(\sum \limits_{m=1}^{\infty}\frac{s}{m(s+m)}-\sum \limits_{m=1}^{\infty}\frac{s-j}{m(s-j+m)}\bigg),\nonumber
\end{align}where $\Gamma(s+1)$ is the Gamma function defined by $\Gamma(s):=\int \limits_{0}^{\infty}e^{-t}t^{s-1}dt$ and $\gamma =\lim \limits_{n\longrightarrow \infty}\bigg(\sum \limits_{k=1}^{n}\frac{1}{k}-\log n\bigg)=0.577215664\cdots $ is the Euler-Mascheroni constant.
\end{theorem}

\begin{proof}
Let $f(t)$ be a real-valued function, contineously differentiable on the interval $[0,\infty)$ and $f(t)\geq 1$ for all $t\in [0,\infty)$. Then we set \begin{align}F(s):=\int \limits_{1}^{s}f(t)\bigg(\log f(t)\bigg)^{s}dt \nonumber 
\end{align}for $s>1$. In the simplest case, we choose $f(t)=e^t$, since it satisfies the hypothesis. Thus $F(s)=\int \limits_{1}^{s}e^tt^sdt$. By application of integration by parts, we find that $F(s):=\int \limits_{1}^{s}e^tt^s=e^ss^s-se^ss^{s-1}+s(s-1)e^ss^{s-2}-s(s-1)(s-2)e^ss^{s-3}+s(s-1)(s-2)(s-3)e^ss^{s-4}+I(s)+\beta(s)$, where $I(s)$ and $\beta(s)$ are convergent. More precisely, we can write $F(s)$ in a closed form as\begin{align}F(s)=\sum \limits_{j=0}^{\lfloor s\rfloor}(-1)^{j}e^ss^{s-j}\frac{\Gamma(s+1)}{\Gamma(s+1-j)}+e\sum \limits_{j=0}^{\lfloor s\rfloor}(-1)^{j+1}\frac{\Gamma(s+1)}{\Gamma(s+1-j)}\nonumber \\+\bigg(\frac{1-((-1)^{s-\lfloor s\rfloor +2})^{1/(s-\lfloor s\rfloor +2)}}{2}\bigg)\bigg(\sum \limits_{j=\lfloor s\rfloor+1}^{\infty}(-1)^{j}e^ss^{s-j}\frac{\Gamma(s+1)}{\Gamma(s+1-j)}\nonumber \\+e\sum \limits_{j=\lfloor s\rfloor+1}^{\infty}(-1)^{j+1}\frac{\Gamma(s+1)}{\Gamma(s+1-j)}\bigg).\nonumber
\end{align} Now, since $\Gamma(s)$ is analytic in the half plane $\mathrm{Re}(s)\geq 1$, it follows by the convergence of $F(s)$ that\begin{align}F'(s)=e^{s}s^s\Gamma(s+1)\sum \limits_{j=0}^{\lfloor s\rfloor}\frac{(-1)^j}{s^j\Gamma(s+1-j)}+e^ss^s(\log s+1)\Gamma(s+1)\nonumber \\\sum \limits_{j=0}^{\lfloor s\rfloor}\frac{(-1)^j}{s^j\Gamma(s+1-j)}+e^{s}s^s\Gamma'(s+1)\sum \limits_{j=0}^{\lfloor s\rfloor}\frac{(-1)^j}{s^j\Gamma(s+1-j)}+e^ss^s\Gamma(s+1)\nonumber \\\sum \limits_{j=0}^{\lfloor s\rfloor}(-1)^{j+1}\frac{js^{j-1}\Gamma(s+1-j)+s^{j}\Gamma'(s+1-j)}{s^{2j}\Gamma^2(s+1-j)}\nonumber \\+\bigg(\frac{1-((-1)^{s-\lfloor s\rfloor +2})^{1/(s-\lfloor s\rfloor +2)}}{2}\bigg)\left \{e^{s}s^s\Gamma(s+1)\sum \limits_{j=\lfloor s\rfloor +1}^{\infty}\frac{(-1)^j}{s^j\Gamma(s+1-j)}\right \}\nonumber \\+\bigg(\frac{1-((-1)^{s-\lfloor s\rfloor +2})^{1/(s-\lfloor s\rfloor +2)}}{2}\bigg)e^ss^s(\log s+1)\Gamma(s+1)\sum \limits_{j=\lfloor s\rfloor +1}^{\infty}\frac{(-1)^j}{s^j\Gamma(s+1-j)}\nonumber \\+\bigg(\frac{1-((-1)^{s-\lfloor s\rfloor +2})^{1/(s-\lfloor s\rfloor +2)}}{2}\bigg)e^{s}s^s\Gamma'(s+1)\sum \limits_{j=\lfloor s\rfloor +1}^{\infty}\frac{(-1)^j}{s^j\Gamma(s+1-j)}\nonumber \\+\bigg(\frac{1-((-1)^{s-\lfloor s\rfloor +2})^{1/(s-\lfloor s\rfloor +2)}}{2}\bigg)e^ss^s\Gamma(s+1)\nonumber \\\sum \limits_{j=\lfloor s\rfloor +1}^{\infty}(-1)^{j+1}\frac{js^{j-1}\Gamma(s+1-j)+s^{j}\Gamma'(s+1-j)}{s^{2j}\Gamma^2(s+1-j)}\nonumber \\+\Gamma'(s+1)\sum \limits_{j=0}^{\lfloor s\rfloor}(-1)^{j+1}\frac{1}{\Gamma(s+1-j)}+\Gamma(s+1)\sum \limits_{j=0}^{\lfloor s\rfloor}(-1)^{j+2}\frac{\Gamma'(s+1-j)}{\Gamma^2(s+1-j)}\nonumber \\+\bigg(\frac{1-((-1)^{s-\lfloor s\rfloor +2})^{1/(s-\lfloor s\rfloor +2)}}{2}\bigg)\bigg(\Gamma'(s+1)\sum \limits_{j=\lfloor s\rfloor +1}^{\infty}(-1)^{j+1}\frac{1}{\Gamma(s+1-j)}\nonumber \\+\Gamma(s+1)\sum \limits_{j=\lfloor s\rfloor +1}^{\infty}(-1)^{j+2}\frac{\Gamma'(s+1-j)}{\Gamma^2(s+1-j)}\bigg).\nonumber
\end{align}On the other hand $F'(s)=e^ss^s$. Arranging terms and comparing both results we find that \begin{align}1=\sum \limits_{j=0}^{\lfloor s\rfloor}(-1)^{j}\bigg(\frac{\Gamma(s+1)}{s^{j}\Gamma(s+1-j)}+\frac{(\log s+1)\Gamma(s+1)}{s^j\Gamma(s+1-j)}+\frac{\Gamma'(s+1)}{s^j\Gamma(s+1-j)}\nonumber \\-\frac{j\Gamma(s+1)}{s^{j+1}\Gamma(s+1-j)}-\frac{\Gamma(s+1)\Gamma'(s+1-j)}{s^j\Gamma^2(s+1-j)}\bigg)\nonumber \\+\bigg(\frac{1-((-1)^{s-\lfloor s\rfloor +2})^{1/(s-\lfloor s\rfloor +2)}}{2}\bigg)\sum \limits_{j=\lfloor s\rfloor +1}^{\infty}(-1)^{j}\bigg(\frac{\Gamma(s+1)}{s^{j}\Gamma(s+1-j)}\nonumber \\+\frac{(\log s+1)\Gamma(s+1)}{s^j\Gamma(s+1-j)}+\frac{\Gamma'(s+1)}{s^j\Gamma(s+1-j)}\nonumber \\-\frac{j\Gamma(s+1)}{s^{j+1}\Gamma(s+1-j)}-\frac{\Gamma(s+1)\Gamma'(s+1-j)}{s^j\Gamma^2(s+1-j)}\bigg)\nonumber\\ +\frac{1}{e^{s-1}s^s}\Gamma'(s+1)\sum \limits_{j=0}^{\lfloor s\rfloor}(-1)^{j+1}\frac{1}{\Gamma(s+1-j)}+\frac{1}{e^{s-1}s^s}\Gamma(s+1)\sum \limits_{j=0}^{\lfloor s\rfloor}(-1)^{j+2}\frac{\Gamma'(s+1-j)}{\Gamma^2(s+1-j)}\nonumber \\+\bigg(\frac{1-((-1)^{s-\lfloor s\rfloor +2})^{1/(s-\lfloor s\rfloor +2)}}{2}\bigg)\bigg(\frac{1}{e^{s-1}s^s}\Gamma'(s+1)\sum \limits_{j=\lfloor s\rfloor +1}^{\infty}(-1)^{j+1}\frac{1}{\Gamma(s+1-j)}\nonumber \\+\frac{1}{e^{s-1}s^s}\Gamma(s+1)\sum \limits_{j=\lfloor s\rfloor +1}^{\infty}(-1)^{j+2}\frac{\Gamma'(s+1-j)}{\Gamma^2(s+1-j)}\bigg)
\end{align}Using the following identities involving the Gamma function \cite{sebah2002introduction}\begin{align}\frac{1}{\Gamma(s+1-j)}:=e^{\gamma (s-j)}\prod \limits_{m=1}^{\infty}\bigg(1+\frac{s-j}{m}\bigg)e^{-(s-j)/m}
\end{align}\begin{align}\frac{\Gamma'(s+1)}{\Gamma(s+1)}:=-\gamma+\sum \limits_{m=1}^{\infty}\frac{s}{m(s+m)},
\end{align}the remaining task is to arrange the terms and apply these identities and identify the function $\eta_{s}(j)$ and $\alpha_{s}(j)$. We leave the remaining task to the reader to verify. 
\end{proof}

\begin{remark}
Now we examine some immediate conequences of the above result, in the following sequel.
\end{remark}

\begin{corollary}
The identity \begin{align}\sum \limits_{j=0}^{\infty}(-1)^j\bigg(\frac{2}{3}\bigg)^j\eta_{3/2}(j)+\frac{1}{(3/2)^{3/2}\sqrt{e}}\sum \limits_{j=0}^{\infty}(-1)^{j+1}\alpha_{3/2}(j)=\frac{4}{3\sqrt{\pi}}, \nonumber
\end{align}where \begin{align}\eta_{3/2}(j)=\bigg(e^{\gamma (3/2-j)}\prod \limits_{m=1}^{\infty}\bigg(1+\frac{3/2-j}{m}\bigg)e^{-(3/2-j)/m}\bigg)\bigg(2+\log \bigg(3/2\bigg)-\frac{2j}{3}\nonumber \\+\frac{3}{2}\sum \limits_{m=1}^{\infty}\frac{1}{m(m+3/2)}-\sum \limits_{m=1}^{\infty}\frac{\frac{3}{2}-j}{m(m+3/2-j)}\bigg),\nonumber
\end{align}and \begin{align}\alpha_{3/2}(j):=\bigg(e^{\gamma (3/2-j)}\prod \limits_{m=1}^{\infty}\bigg(1+\frac{3/2-j}{m}\bigg)e^{-\frac{3/2-j}{m}}\bigg)\bigg(\sum \limits_{m=1}^{\infty}\frac{3/2}{m(3/2+m)}\nonumber \\-\sum \limits_{m=1}^{\infty}\frac{3/2-j}{m(3/2-j+m)}\bigg),\nonumber
\end{align}remains valid.
\end{corollary}

\begin{proof}
Let us set $s=\frac{3}{2}$ in Theorem \ref{universal}. Then it follows that \begin{align}\sum \limits_{j=0}^{\infty}(-1)^{j}\bigg(\frac{2}{3}\bigg)^j\eta_{3/2}(j)+\frac{1}{(3/2)^{3/2}\sqrt{e}}\sum \limits_{j=0}^{\infty}(-1)^{j+1}\alpha_{3/2}(j)=\frac{1}{\Gamma(5/2)}=\frac{4}{3\Gamma (1/2)}=\frac{4}{3\sqrt{\pi}}, \nonumber
\end{align}where we have used the relation $\Gamma(\frac{1}{2})=\sqrt{\pi}$ \cite{sebah2002introduction}. The proof is completed by computing $\eta_{3/2}(j)$ and $\alpha_{3/2}(j)$  given in Theorem \ref{universal}.
\end{proof}

\begin{corollary}
The identity \begin{align}\sum \limits_{j=0}^{\infty}(-1)^j\bigg(\frac{3}{5}\bigg)^j\eta_{5/3}(j)+\frac{1}{(5/3)^{5/3}\sqrt[3]{e^2}}\sum \limits_{j=0}^{\infty}(-1)^{j+1}\alpha_{5/3}(j)=\frac{9}{10\Gamma(2/3)},\nonumber
\end{align}is valid, where \begin{align}\eta_{5/3}(j)=\bigg(e^{\gamma (5/3-j)}\prod \limits_{m=1}^{\infty}\bigg(1+\frac{5/3-j}{m}\bigg)e^{-(5/3-j)}\bigg)\bigg(2+\log \bigg(5/3\bigg)\nonumber \\-\frac{3j}{5}+\sum \limits_{m=1}^{\infty}\frac{5/3}{m(5/3+m)}-\sum \limits_{m=1}^{\infty}\frac{5/3-j}{m(5/3-j+m)}\bigg)\nonumber
\end{align}and \begin{align}\alpha_{5/3}(j):=\bigg(e^{\gamma (5/3-j)}\prod \limits_{m=1}^{\infty}\bigg(1+\frac{5/3-j}{m}\bigg)e^{-\frac{5/3-j}{m}}\bigg)\bigg(\sum \limits_{m=1}^{\infty}\frac{5/3}{m(5/3+m)}\nonumber \\-\sum \limits_{m=1}^{\infty}\frac{5/3-j}{m(5/3-j+m)}\bigg),\nonumber
\end{align}
\end{corollary}

\begin{proof}
The result follows by setting $s=\frac{5}{3}$ in Theorem \ref{universal}, and computing $\eta_{5/3}(j)$ and $\alpha_{5/3}(j)$.
\end{proof}

\begin{corollary}
For any integer $s\geq 2$, the inequality \begin{align}\left|\sum \limits_{j=0}^{s}\frac{(-1)^j\eta_{s}(j)}{s^j}\right|-\frac{1}{e^{s-1}s^s}\left|\sum \limits_{j=0}^{s}(-1)^{j}\alpha_{s}(j)\right|\leq \frac{1}{\Gamma(s+1)}<\sum \limits_{j=0}^{s}\frac{|\eta_{s}(j)|+|\alpha_{s}(j)|}{s^j},\nonumber
\end{align}where \begin{align}\eta_s(j):=\bigg(e^{\gamma (s-j)}\prod \limits_{m=1}^{\infty}\bigg(1+\frac{s-j}{m}\bigg)\nonumber \\e^{-(s-j)/m}\bigg)\bigg(2+\log s-\frac{j}{s}+\sum \limits_{m=1}^{\infty}\frac{s}{m(s+m)}-\sum \limits_{m=1}^{\infty}\frac{s-j}{m(s-j+m)}\bigg) \nonumber
\end{align}and \begin{align}\alpha_s(j):=\bigg(e^{\gamma (s-j)}\prod \limits_{m=1}^{\infty}\bigg(1+\frac{s-j}{m}\bigg)e^{-(s-j)/m}\bigg)\bigg(\sum \limits_{m=1}^{\infty}\frac{s}{m(s+m)}-\sum \limits_{m=1}^{\infty}\frac{s-j}{m(s-j+m)}\bigg),\nonumber
\end{align}is valid.
\end{corollary}

\begin{proof}
If $s\geq 2$ is an integer, then Theorem \ref{universal} reduces to \begin{align}\sum \limits_{j=0}^{ s}\frac{(-1)^j}{s^j}\eta_s(j)-\frac{1}{e^{s-1}s^s}\sum \limits_{j=0}^{s}(-1)^{j}\alpha_s(j)=\frac{1}{\Gamma(s+1)},\nonumber
\end{align} and the result follows immediately by applying the triangle inequality.
\end{proof}
%

\bibliographystyle{amsplain}

\end{document}